\documentclass{amsart}
\usepackage[utf8]{inputenc}
\usepackage[margin=1in]{geometry}
\usepackage{upgreek,amsthm,bbm,bm}
\usepackage{color,amssymb, comment, mathrsfs,tikz,amsmath,amsfonts,stmaryrd,tikz-cd,hyperref,appendix,cancel,faktor,verbatim}
\usepackage{amsmath,amsthm,amssymb}
\usepackage{tikz,tikz-cd,color}
\usetikzlibrary{arrows,chains,matrix,positioning,scopes}
\usepackage{enumitem}
\usepackage{mathtools}
\usepackage{comment}

\usepackage[all,cmtip]{xy} 
\xyoption{all}
\xyoption{arc}

\usepackage{amsmath, nccmath}
\usepackage{geometry}

\hypersetup{
    colorlinks=true,
    linkcolor=blue,
    filecolor=magenta,      
    urlcolor=cyan,
    citecolor=purple,
} 

\usepackage[capitalise]{cleveref}
\crefformat{equation}{(#2#1#3)}
\crefrangeformat{equation}{(#3#1#4--#5#2#6)}
\crefformat{enumi}{(#2#1#3)}
\crefrangeformat{enumi}{(#3#1#4--#5#2#6)}

\newtheorem{theorem}{Theorem}[section]
\newtheorem*{main*}{Main Theorem}
\newtheorem{lem}[theorem]{Lemma}
\newtheorem{prop}[theorem]{Proposition}
\newtheorem{corollary}[theorem]{Corollary}
\theoremstyle{definition}
\newtheorem{case}[theorem]{Case}
\newtheorem{defn}[theorem]{Definition}

\newtheorem{ex}[theorem]{Example}

\newcounter{intro}

\theoremstyle{remark}
\newtheorem{remark}[theorem]{Remark}

\newcommand{\Tor}{\ensuremath{\operatorname{Tor}}}

\usepackage[maxbibnames=9,style=alphabetic]{biblatex}
\title{Relating the Exterior Algebra of the Conormal Module to the Tor Algebra}
\date{}

\usepackage[capitalise]{cleveref}
\crefformat{equation}{(#2#1#3)}
\crefrangeformat{equation}{(#3#1#4--#5#2#6)}
\crefformat{enumi}{(#2#1#3)}
\crefrangeformat{enumi}{(#3#1#4--#5#2#6)}

\author[D. Martin]{Desiree Martin}
\address{Department of Mathematics,
Syracuse University, Syracuse, NY 13244, U.S.A.}
\email{dmarti02@syr.edu}

\begin{document}

\begin{abstract}
For any ideal $I$ (whose projective dimension need not be finite) in a local Noetherian ring $R$, we show that, if the conormal module $I/I^2$ has a free summand given by $F$, the natural map from $\bigwedge F$ to $\Tor_{*}^{R}(R/I,R/I)$ splits as algebras. To achieve this, we use dg algebra techniques and  remodel results of André and Iyengar, proving them in setting of semi-free extensions, replacing the need for semi-free $\Gamma$-extensions. In this new setting, we show that free summands of the conormal module correspond to central elements of the relative homotopy Lie algebra $\pi^*(\varphi)$ and, lastly, we provide an explicit splitting morphism in lower degrees.

\end{abstract}

\maketitle

\section{Introduction}
\label{sec:intro}

Let $R$ be a Noetherian local ring and $I$ an ideal of $R$. An important invariant studied in commutative algebra and algebraic geometry is the conormal module $I/I^2$. In algebraic geometry, one uses the conormal module capture information about the tangent behavior of a subvariety that is embedded within a larger variety. On the commutative algebra side of the story, the conormal module plays a role in characterizing types of singularities as well as ring theoretic properties of $S=R/I$. 

In 1967 Ferrand \cite{Ferrand} and Vasconcelos \cite{vasconcelos1967ideals} proved that if $I$ has finite projective dimension and the conormal module $I/I^2$ is free as an $S$-module then $I$ is a complete intersection ideal, providing us with some conditions that $I/I^2$ and $I$ must satisfy in order to show that $I$ is generated by a regular sequence. Vasconcelos then conjectured that if one weakens the condition of freeness on $I/I^2$ to one of finite projective dimension, then the same conclusion holds. Avramov and Herzog \cite{avramov1994jacobian} were able to prove the conjecture in the graded case, but in the nongraded case the conjecture was still open until more recently when Briggs proved Vasconcelos's Conjecture in \cite{Briggs}, finally putting the investigation of the conormal module when $I$ has finite projective dimension to rest. In this paper, we study the conormal module when $I$ has infinite projective dimension, as the story is more mysterious.

Moreover, when $I$ is generated by a regular sequence, it is well known that $I/I^2$ is free and that  $\Tor^{R}_{*}(S,S)$ and  the exterior algebra of the conormal module $\bigwedge I/I^2$ are isomorphic as algebras. This isomorphism seems to imply a relationship between the underlying algebra structures of $\Tor^{R}_*(S,S)$  and $\bigwedge I/I^2$. Indeed, in \cite{GotoSuzuki}, Goto and Suzuki shed some light on the relationship, showing that $\Tor^R_2(S,S)$ is $S$-free if and only if both $I/I^2$ and the first Koszul homology $H_1(I)$ of $I$ are $S$-free. Moreover, they showed under these assumptions that $\Tor^R_2(S,S) = \bigwedge^2 (I/I^2) \oplus H_1(I)$. In this paper, we continue the study of the relationship between the algebra $\Tor_{*}^{R}(S,S)$ and $S$-module $I/I^2$. In doing so, we drop the assumption that the individual $\Tor$ modules and the first Koszul homology are free. 
 
 With these conditions, we have the following theorem:

\begin{main*}[Theorem \ref{maintheoremSec3}] 
    Let $R$ be a Noetherian local ring and $I$ be an ideal of $R$ with $F$  a free summand of $I/I^2$ as an $S= R/I$-module. Then the natural map from $\bigwedge F$ to $\Tor^{R}(S,S)$ splits as a map of algebras. 
\end{main*}

In particular, if $I/I^2$ is $S$-free, we see that the exterior algebra of the conormal module $\bigwedge I/I^2$ splits completely from $\Tor_{*}^{R}(S,S)$ as algebras. 

Another major goal of this paper is to give an explicit splitting morphism from the Tor algebra to the exterior algebra of the conormal module. However, the natural setting for the splitting morphism uses polynomial variables in even degrees, as can be seen from the explicit splitting in low degrees developed in \cref{sec:explicit-maps}. This inspires us to rework results of André and Iyengar, whose settings involve semi-free $\Gamma$-extensions, which use divided power variables in even degrees, and instead, change to semi-free extensions. Further, in this setting we are able to prove that free summands of the conormal module correspond to central elements of the relative homotopy Lie algebra $\pi^*(\varphi)$, where $\varphi: R \twoheadrightarrow S$, see \cref{centralrecovery}.

We now give an outline of the sections. In Section \ref{sec:background} we recall some key concepts related to differential graded (dg) algebras and semi-free ($\Gamma$-)extensions. In Section \ref{sec:minimal-model}, we reexamine results of André in \cite{André1982} and Iyengar in \cite{Iyengar2001} and convert them to the setting of semi-free algebras/minimal models rather than semi-free $\Gamma$-algebras/acyclic closures. In Section \ref{sec:existence} we prove the existence of a splitting morphism on the level of dg algebras which induces the splitting morphism in \cref{maintheoremSec3}, and provide conclusions for the relative homotopy Lie algebra, $\pi^*(\varphi)$. In an attempt to paint a broader picture, we also provide examples where $\bigwedge I/I^2$ has no natural injection into $\Tor_{*}^{R}(S,S)$ when $I/I^2$ is not $S$-free. Finally, in the more natural setting of semi-free extensions, we are able to construct an explicit splitting morphism of dg algebras in low degrees in Section \ref{sec:explicit-maps} and describe an algorithm that one may use for higher degrees.

\section{Background}
\label{sec:background}

The maps in this paper are defined using differential graded algebras as well as semi-free extensions/minimal models. We begin with definitions of the specific machinery needed.

\begin{defn}
A \textit{differential graded (dg) algebra} over a ring $R$ is a complex $(A,\partial^A)$ of free $R$-modules equipped with a unitary, associative multiplication $A \otimes_R A \to A$ satisfying 
\begin{itemize}
    \item[(i)] $A_iA_j \subseteq A_{i+j}$,
    \item[(ii)] $a_ia_j=(-1)^{ij}a_ja_i$,
    \item[(iii)] $a_i^2=0 \text{ if $i$ is odd}$,
    \item[(iv)] $\partial^A(a_ia_j)=\partial^A(a_i)a_j+(-1)^{i}a_i\partial^A(a_j)$,
\end{itemize}
where $a_\ell \in A_\ell$.  Conditions (i)-(iii) are together called \textit{graded commutativity}, and condition (iv) is called the \textit{Leibniz rule}. We also require $A_0$ be a Noetherian $R$-algebra and each $A_i$ to be a finitely generated $A_0$-module. A dg algebra is called a \emph{dg $\Gamma$-algebra} if it admits a system of divided powers. For the interested reader, see (\cite{AvramovIFR}) for the definition of divided power variables. 
\end{defn}

While the results in this paper uses semi-free extensions over semi-free $\Gamma$-extensions, we will refer to older results that use this setting.

\begin{defn}[\cite{tate1957homology}, see also \cite{AvramovIFR}]
\label{def-dga}
    A \emph{semi-free $\Gamma$-extension} $A\langle X\rangle$ of a dg algebra $A$ is a dg algebra obtained by successive adjunction of sets of variables $X=X_0,X_1, \ldots$, where $X_i=\{x \in X \mid |x|=i\}$,  which have exterior variables in odd degrees, divided power variables in positive even degrees, and polynomial variables in degree 0.

     We write $A\langle X_{\leq i} \rangle$ for the dg subalgebra of $A\langle X\rangle$ that includes the adjunction of variables up to degree $i$.

    \end{defn}

        Note that since a semi-free $\Gamma$-extension is defined to be a dg algebra, one automatically has $\partial(X_i)$ are cycles in $A\langle X_0, \ldots, X_{i-1}\rangle$. 

\begin{defn}

    If $S=R/I$ then a \textit{semi-free $\Gamma$-resolution} $R\langle X \rangle $ of $S$ is a semi-free $\Gamma$-extension where the images of $\partial(X_i)$ generate $H_{i-1}\left(R\langle X_1, \ldots, X_{i-1}\rangle \right)$ for  $i> 1$ and the images of $\partial(X_1)$ generate the kernel of the augmentation map $R \to S$. Equivalently, $R\langle X \rangle  \to S$ is a quasi-isomorphism. If, in addition, $X$ is chosen minimally, we call $R\langle X \rangle$ the \emph{acyclic closure} of $R$. Note that we use ``the" instead of ``an", as any two acyclic closures of $R$ are isomorphic. 
\end{defn}

Since the map $\varphi: R \to S$ is surjective, we do not need to adjoin variables of degree $0$ when constructing the acyclic closure of $S$ over $R$.

\begin{defn}
       A \textit{semi-free extension} of a dg algebra $A$ is a dg algebra $A[X]$ obtained by repeated adjunction of sets of free variables $X= X_0, X_1, \ldots$, of exterior type in odd degree and polynomial type in even degree and we write $A[ X_{\leq i} ]$ for the dg subalgebra of $A[ X]$ that includes the adjunction of variables up to degree $i$.
    
\end{defn}

\begin{defn}
    A \textit{minimal} semi-free extension of $R$ is a semi-free extension $R \hookrightarrow{} R[Y]$ such that $Y=Y_{\geq1}$, and the differential $\partial$ is decomposable in the sense that \[  \partial(Y) \subseteq 
    (Y,m_R)^2.\]
If $S=R/I$ with $I \subseteq m_R^2$, then a \textit{minimal model} of $S$ over $R$ is a quasi-isomorphism $R[Y] \to S$, where $R[Y]$ is a minimal semi-free extension of $R$. 
\end{defn}

For convenience of the reader, we end this section by recalling the definition of the Koszul complex as an exterior algebra.

\begin{defn}
    Let $R$ be a ring and $M$ a free $R$-module minimally generated by $e_1, \ldots, e_n$ with the $R$-module map $\varphi: M \to R$. The \textit{Koszul complex} is 
    \[
    K(\varphi): 0 \xrightarrow{} \bigwedge ^n M \xrightarrow{\partial_n} \bigwedge^{n-1}M \xrightarrow{} \cdots \xrightarrow{} \bigwedge^2M \xrightarrow{} M \xrightarrow{\varphi} R \xrightarrow{} 0,
    \]
    with differential given by $$\partial_m(e_{i_1}\wedge \ldots \wedge e_{i_m}) = \sum_{j} (-1)^{j+1} \varphi(e_{i_j})e_{i_1}\wedge \cdots \wedge \hat{e}_{i_j}\wedge \cdots \wedge e_{i_m}.$$ We can realize the Koszul complex as the exterior algebra
    \[K(\varphi) = \bigwedge M = \bigoplus_j M_j = R\langle e_1, \ldots, e_n \rangle,\] which is a dg algebra.
\end{defn}

\section{ Reformulation in Terms of the Minimal Model}
\label{sec:minimal-model}

In this section, we show that results analogous to those of Andr\'e and Iyengar hold in the setting of semi-free extensions rather than semi-free $\Gamma$-extensions. We do this since using polynomial variables instead of divided power variables allows for a more natural setting, as for example, when defining the explicit splitting morphism in Section \ref{sec:explicit-maps}.

We now establish the setting for the rest of the paper. Let $R$ be a Noetherian local ring and assume $I$ is an ideal of $R$ minimally generated by $a_1,\ldots,a_n, b_1, \ldots, b_m$ such that the classes of $a_1,\ldots, a_n$ form a basis for a free summand $F$ of $I/I^2$ as an $R/I$-module and the classes of $b_1,\ldots, b_m$ are a minimal generating set for a complement to $F$. Equivalently, the generating set satisfies condition $1.3.2$ in \cite{Iyengar2001}:
for any $r_1,\ldots, r_n$ in $R$ 
\begin{equation}
\label{Freebasis cond}
\sum_{i=1}^n r_ia_i \in (b_1,\ldots,b_m) \implies r_i \in I \quad \text{for} \quad 1 \leq i\leq n.
\end{equation} In this case, there is always a natural map $\bigwedge F \to \Tor^R_{*}(R/I,R/I)$, which we recall below.

We first define $\bigwedge F$ and $\Tor_{*}^{R}(R/I,R/I)$ as the homology of dg algebras. 
Let $R[ X ]$ be the minimal model of $S=R/I$ over $R$. Note that one can choose a set of variables of degree one $X_1$ such that there is a subset $X'_1=\{x_1,\ldots, x_n\}$ of $X_1$ where $\partial(x_i)=a_i$ for each $i \in \{1, \ldots, n
\}$. We can use the minimal model to define $\Tor_{i}^{R}(S,S)$, indeed for all $i \geq 0$,
\[
\Tor_{i}^{R}(S,S) 
= H_i(R[ X ]\otimes_{R} S)
\cong H_i(S[ X ]).
\]
 Since $R[ X_1' ] = R\langle X'_1\rangle$ is the Koszul complex on $a_1,\ldots,a_n$, the differential on $S[ X_1' ]$ is zero. The conormal module $I/I^2$ has a free summand $F$ with basis $\overline{a_1}, \ldots, \overline{a_n}$, and so we can identify the following modules via the isomorphism that sends $x_i$ to $a_i$:
\[
H(S[ X_1' ] )
= S[ X_1' ]
\cong \bigwedge F.
\]
The natural inclusion map $j: S[X_1'] \to S[ X ]$ gives the following map on homology algebras :
\[
H(j) \colon  \bigwedge F \to \Tor^{R}_{*}(S,S).
\]

We want to show $j$ is a retraction of algebras, which in turn shows $H(j)$ is split. To do this, we obtain a dg algebra map $\Psi \colon S[ X] \to S[ X_1' ]$ which uses the following result, remodeled from \cite{Iyengar2001}. Before stating the proposition, recall that the free rank of a finitely generated module $M$ over a ring $S$ is the maximal possible rank of a free summand and is denoted f-rank$_{S}(M)$.


\begin{prop}
\label{prop-minimal}
    Let $I$ be a proper ideal of a ring $R$, and let $R[X]$ be the minimal model the map $\phi: R\to S=R/I$. Let $n = \text{f-rank}_R(I/I^2)$. If $A$ is a dg algebra over $R$ such that $IA_0\subset \partial(A_1)$, then $$A[ X ] = A \otimes_{R} R[ X ] \cong W[ z_1, \ldots, z_n | \partial(z_i)=0],$$
    where $W =  \bigcap_{x \in X_1'}\text{ker} \left(\theta_x\right)$ is a DG subalgebra of $A[ X ]$.
\end{prop}

We will not provide a detailed proof of this proposition, as it is similar to the proof of Proposition 2.1 given by Iyengar in \cite{Iyengar2001}. However, we show the necessary pieces of Iyengar's proof hold in the setting of semi-free extensions. Namely, we show the extension properties for the derivations hold in the setting of non $\Gamma$-derivations and show previous results of André also hold when using semi-free extensions in place of semi-free $\Gamma$-extensions. We begin with the definition of a derivation.

 \begin{defn}
     An $R$-linear \textit{derivation} of a graded $R$-algebra $A$ is a homogeneous $R$-linear map $\theta: A \to A$ such that
     \[
     \theta(xy) = \theta(x)y+(-1)^{|\theta||x|}x\theta(y)  \hspace{2mm}\text{for} \hspace{2mm} \text{homogeneous} \, \, x,y\in A.
     \]
     It is called a \textit{chain derivation} if it is also a chain map.
    
  \end{defn}

Note, these derivations are not $\Gamma$-derivations. We next look at a version of the extension property from \cite{GulliksenLevin} for these derivations.

  \begin{lem}
  \label{Extension Lemma}
      Let $A$ be a dg $R$-algebra and let $A[X]$ be a semi-free extension with $\partial(x) \in A$ for all $x \in X$. 
      A chain derivation $\theta$ on $A$  extends to a chain derivation $\widetilde{\theta}$ on $A[X]$ with $\widetilde{\theta}(X) \subseteq A$ if and only if $\theta(\partial(x))\in \partial(A)$ for each $x \in X$. If for each $x$ one has $\theta(\partial (x))= \partial(b_x)$ for some $b_x\in A$, then one may extend $\theta$ to $A[X]$ by $\widetilde{\theta}(x)=(-1)^{|\theta|}b_x$.

      \end{lem}

\begin{proof}
First assume $\theta(\partial(x)) \in \partial(A)$ and assign $\widetilde{\theta}$ as in the statement, i.e. $\theta(\partial (x))=
\partial(b_x)$ for $b_x\in A$, so set $\widetilde{\theta}(x)=(-1)^{|\theta|}b_x$. 
Assigning $\widetilde{\theta}(x)$ for each $x \in X$ as in the statement above, we get that $\theta$ automatically extends to a derivation, which we will call $\widetilde{\theta}$, of the semi-free extension $A[X]$, since $A[X]$ is a free graded commutative algebra over $A$. We need to show that given this assignment, this extension is a chain map. Since $\widetilde{\theta}$ is a derivation on $A[X]$, as well as a chain map on $A$, a quick computation using the Leibniz rule will show that $\widetilde{\theta}$ commutes with $\partial$ on products and hence powers of variables.

For the other direction, we assume $\theta$ is a chain derivation on $A$ that extends to a derivation $\widetilde{\theta}$ on $A[X]$ with $\widetilde{\theta}(x)\in A$ for all $x \in X$ and we want to show $\theta(\partial(x))\in \partial(A)$ for each $x\in X$. This follows directly from the fact that $\theta$ is a chain derivation on $A$ since,
\[
\theta(\partial(x)) = \widetilde{\theta}(\partial(x)) = (-1)^{|\theta|}\partial(\widetilde{\theta}(x))
\]
and $\widetilde{\theta}(x) \in A$.
\end{proof}

We now construct the derivations $\theta_x$ needed in the proof of Iyengar's Propositon 2.1. 
 \begin{lem}\label{extn lemma2}
     Let $R[X]$ be the minimal model of the quotient map $R \to R/I$. There exist chain derivations $\theta_x: R[X] \to R[X]$ for each $x \in X_1'$ so that $\theta_x(x)=1$ and $\theta_x(y) = 0$ for all $y$ with $|y|\leq|x|$.
  \end{lem}

\begin{proof}
   Let $x \in X_1'$, we build $\theta_{x}$ by induction on homological degree $i$. For $i =1$, $R[X_1]=R\langle X_1 \rangle$, so we have the existence of the standard chain derivations (of degree $-1$) $\theta^1_{x_j}: R[X_1] \to R[X_1]$ for any $x_j \in X_1$ where $\theta_{x_j}(x_k)=\delta_{jk}$ (see \cite{AvramovIFR}); we will consider only the derivations $\theta^1_{x}$, where $x \in X'_1 \subseteq X_1$.

  Suppose $\theta_{x}^{i}: R[X_{\leq i}] \to R[X_{\leq i}]$ is an $R$-linear chain derivation. We want to show $\theta^{i}_{x}(\partial X_{i+1}) \subseteq \partial R[X_{\leq i}]$, so that $\theta^{i}_{x}$ extends to \[\theta^{i+1}_{x}: R[X_{\leq i+1}] \to R[X_{\leq i+1}],\] using Lemma \ref{Extension Lemma} with $A= R[X_{\leq i}]$. Once we prove we can extend, we then define $\theta_{x} := \text{colim}_{i} \theta_{x}^{i}$ and finish the proof.

  If $i=1$, then for $y \in X_2$,
  \[
  0= \partial^2(y) = \sum_{x_j\in X_1'} r_jx_j + \sum_{x_k \notin X_1'} r_kx_k = \sum_j r_ja_j + \sum_k r_kb_k.
  \]
  This implies $r_j \in I$ for each $j \in \{1, \ldots, n\}$ by \ref{Freebasis cond}, and since $\theta_x^1(x_i)=1$ if $x_i \in X_1'$ and $0$ otherwise, we have $\theta_{x}^{1}(\partial X_2) \subseteq I \subseteq \partial R[X_1]$.
  
  If $i >1 $, then we know the degree of $\theta_{x}^{i}(\partial X_{i+1})$ is $ i-1$. We also know that \[\theta^{i}_{x}(\partial X_{i+1}) = \partial\left(\theta^{i}_{x}( X_{i+1}) \right) \subseteq \text{ker}(\partial)\]
  and that $R[X_{\leq i }]$ is acyclic in degree $d$ for $0< d \leq i-1$. Thus, $\theta^{i}_{x}(\partial X_{i+1}) \subseteq \partial R[X_{\leq i}]$, and we are done.
\end{proof}

 The last ingredient needed for the proof of \cref{prop-minimal} is a splitting result from André in \cite{André1982}. We begin with the definition of a special cycle. The reader should note the definition is slightly different from the one found in \cite{André1982}, since we are not assuming the derivation is a $\Gamma$-derivation.

 \begin{defn}
     Let $A$ be a dg algebra, respectively a dg $\Gamma$-algebra, and $s$ a cycle of $A$. We say $s$ is a \textit{special cycle} if there exists a chain derivation, respectively a chain $\Gamma$-derivation, $\theta:A\to A$ so that $\theta(s)=1$.
 \end{defn}

 The next result was first noticed by André for a single special cycle (in the setting of $\Gamma-$derivations) and mentioned briefly in \cite{Iyengar2001}. Since a proof for dg $\Gamma$-algebras is not given in \cite{Iyengar2001} we provide the proof for dg $\Gamma$-algebras for the ease of the reader.

\begin{prop}
\label{prop6 ext}
    Let $s_1,\ldots, s_n$ be special cycles of odd degree of the dg $\Gamma$-algebra $A$. Then there exists a dg $\Gamma$-algebra $W$ giving rise to an isomorphism of dg $\Gamma$-algebras $$W\langle z_1, \ldots, z_n \mid \partial(z_i)=0 \rangle \cong A$$ which sends $z_i$ to $s_i$ for  $i \in [n]$.
\end{prop}
\begin{proof}
    For each special cycle $s_i$, there exists a derivation $\theta_i$ of $A$ whose square is zero and sends $s_i$ to 1. We set $W = \bigcap_{i=1}^{n} \ker(\theta_i)$, which is a dg subalgebra of $A$. The inclusion of $W$ into $A$ extends to a homomorphism of dg algebras 
    \[
    \tau: W\langle z_1, \ldots, z_n \mid \partial(z_i)=0 \rangle \to A
    \]
    that  maps each $z_i$ to $s_i$. This yields an isomorphism with inverse
    \[ 
    \sigma: A \to W\langle z_1, \ldots, z_n \mid \partial(z_i)=0 \rangle
    \]
    that sends $a$ to the element
    \[
    \sigma(a)= \underbrace{(a + (-1)^{|a|} \sum_{i}\theta_i(a)s_i)}_{\alpha} - (-1)^{|a|} \sum_i \theta_i(a)z_i,
    \]
    where the expression $\alpha$ is in the kernel of each $\theta_i$.
    It is straightforward to check that $\tau$ and $\sigma$ are inverses.
\end{proof}

  We now claim that the analogous version of Proposition \ref{prop6 ext} is equally valid in the setting of dg algebras. When we consider the setting of Proposition \ref{prop-minimal}, we note the special cycles of interest are degree 1, which is necessary to apply Proposition \ref{prop6 ext} as the variables need to be of odd degree. The proof of Proposition 2.1 (seen in \cite{Iyengar2001}) relies on the isomorphism given in Proposition $\ref{prop6 ext}$ and still holds in the setting of general semi-free extensions as long as one checks that the derivations used extend to the entire semi-free extension $R[X]$, which we did above in \cref{extn lemma2}.

\section{Existence of A Splitting Morphism}
\label{sec:existence}
Recall from the previous section and using the notation there, that there is a natural inclusion map $j: S[X_1'] \to S[ X ]$. In this section we will show this inclusion is a retraction of algebras, which in turn shows \[H(j):\bigwedge F \to \Tor_{*}^{R}(S,S)\] is split. We also provide examples where $I/I^2$ does not have a free summand that show the splitting fails when the necessary conditions are not met.

We begin by constructing a dg algebra map $\Psi \colon S[X] \to S[ X_1' ]$. Applying \cref{prop-minimal} by substituting $S$ for $A$, we obtain an isomorphism  
\begin{equation}
\label{alph}
\alpha \colon S[ X ]\xrightarrow[]{\cong} W[ z_1, \ldots, z_n | \partial(z_i)=0]
\end{equation}
for some subalgebra $W$ of $S[ X ]$.

Since $W$ is a subalgebra of $S[ X ] $, we may embed $W$ into $S[ X ]$ with a map $\iota$. We then project $S[ X ]$ down to $S  = H_{0}(S[ X ]))$ via a morphism of dg algebras we will call $\pi$, which gives us the following composition:
\begin{align*}
    \xymatrixcolsep{1.3pc}
    \xymatrix{&W \ar[rr]^{\iota} & & S[ X ] \ar[rr]^{\pi} & & S.\\}
\end{align*}
The composition $\pi \circ \iota$ then induces a map $\psi': W[ z_1, \ldots, z_n] \xrightarrow[]{} S[ z_1, \ldots,z_n  ]$
by adjoining $z_1, \ldots, z_n$ to the algebras with $\partial(z_i)=0$. Further notice that the differential on $S[ X_1' ]$ is $0$, and so there is an isomorphism $\beta: S[ X_1' ] \to S[ z_1, \ldots z_n | \partial(z_i)=0 ]$ that sends each $x_i \in X_1'$ to $z_i$ for $i \in [n]$.

We then define $\psi$ as the composition $\beta^{-1}\circ \psi'\circ \alpha$ as seen in the diagram below: 
\begin{align}
\label{comm-diag}
\xymatrixcolsep{1.3pc}
\xymatrix{& &S[ X ] \ar[rr]^{\alpha \qquad \qquad} \ar@/^1pc/[dd]^{\Psi= \beta^{-1}\circ \psi'\circ \alpha} &   &{W[ z_1, \ldots, z_n | \partial(z_i)=0]}  \ar[dd]^{\psi'}\\
& & & & & \\ 
&  \left( \bigwedge F, \partial =0 \right) \ar@{-}[r]^{\qquad \cong} & S[  X_{1}' ] \ar[rr]^{\beta \qquad \qquad} \ar@/^/[uu]^{j} & & S[ z_1,\ldots,z_n| \partial(z_i)=0 ]\\}
\end{align}

To show $\Psi$ splits $j$, it is enough to show that it does so on $X_1'$ since $S[ X_1']$ is a free graded commutative $S$-algebra.
Indeed for each $x_i \in X_1'$,
\[
(\Psi\circ j) (x_i)= (\beta^{-1}\circ \psi' \circ \alpha) \circ j(x_i)= (\beta^{-1}\circ \psi' \circ \alpha)(x_i)= (\beta^{-1} \circ \psi')(z_i)= \beta^{-1}(z_i)=x_i.
\]
Thus, $\Psi$ splits $j$ and induces a splitting $H(\Psi)\colon \Tor^R_{*}(S,S) \to \bigwedge F$ of $H(j)$. From the above discussion, we see that $\bigwedge F$ splits out of $\Tor^R(S,S)$ as algebras.

We can further expand this by considering a map of rings given by the composition $R \twoheadrightarrow S \to T$ (where $S \to T$ is any map of rings) and tensoring the objects above with $T$ over $R$. Using the same argument, we have the following theorem.

\begin{theorem}
\label{maintheoremSec3}
    Let $R$ be a Noetherian local ring and $I$ be an ideal of $R$ with $F$ a free summand of $I/I^2$ as an $S= R/I$-module. Furthermore, if $S \to T$ is any map of rings then, the natural map from $\bigwedge F \otimes_{S} T$ to $\Tor^{R}(S,T)$ splits as a map of algebras.
\end{theorem}
Note, in Section \ref{sec:minimal-model} we remodeled some previous results to fit the setting of semi-free extensions, but since those results existed in the setting of semi-free $\Gamma$-extensions we could have proved \cref{maintheoremSec3} in that setting by using Iyengar's Proposition 2.1 in \cite{Iyengar2001} instead of \cref{prop-minimal}. It is also worth noting that the case where $T = k$ is shown in \cite{Kekkou} to obtain a lower
bound on the level of the Koszul complex in the derived category.

If the entire conormal module $I/I^2$ is free, we get the following immediate consequence.

\begin{corollary}
    Let $R$ be a Noetherian local ring, $I$ be an ideal of $R$, and $I/I^2$ be free over $S=R/I$. Then the natural map from $\bigwedge I/I^2$ to $\Tor^{R}_{*}(S,S)$ splits as algebras.
\end{corollary}
While \cref{maintheoremSec3} gives the existence of a splitting map $\psi$, it does not provide an explicit construction. In Section \ref{sec:explicit-maps}, we provide a technique for finding one and show how it works in low degrees. 

In the next corollary we explore some conclusions for the relative homotopy Lie algebra $\pi^*(\varphi)$ that follow from the \cref{prop-minimal}; for the definition, see, for example, 
\cite{AvramovIFR}.
This result appears in \cite{Iyengar2001} for $\pi^*(S)$ using an argument in the setting of semi-free $\Gamma$-extensions, requiring a precise construction of $\Gamma$-derivations to describe the homotopy Lie algebra of the ring $S$, $\pi(S)$. In comparison, we show a similar result (possibly known to experts) for the relative homotopy Lie algebra $\pi^*(\varphi)$ with a shorter proof in the natural setting of semi-free extensions.

\begin{corollary}
\label{centralrecovery}
    Let $R$ be a Noetherian local ring and $I$ be an ideal of $R$ with $F$ a free summand of $I/I^2$ as an $S= R/I$-module. The generators of $F$ correspond to central elements in  $\pi^{*}(\varphi)$, the homotopy Lie algebra of the ring map $\varphi: R \twoheadrightarrow S$.
\end{corollary}

\begin{proof}
    For this proof we keep the same setting as in the paper by assuming $I$ is minimally generated by $a_1,\ldots, a_n, b_1, \ldots, b_m$ such that the classes of $a_1,\ldots, a_n$ form a basis for a free summand $F$ of $I/I^2$ as an $S$-module, and the classes of $b_1, \ldots, b_m$ are a minimal generating set for a complement to $F$. We also recall that $R[X]$ is a minimal model for $S$ such that the set of variables $X_1$ of degree one was chosen such that there is a subset $X_1'= \{x_1, \ldots, x_n\}$ of $X_1$ with $\partial(x_i) =a_i$ for each $i \in \{1, \ldots, n\}$.

    By \cref{prop-minimal}, we have an isomorphism $S[X] \cong W[ z_1,\ldots, z_n \mid \partial(z_i)=0]$, where $W$ is a dg subalgebra of $S[X]$. In particular, this means $\partial(W) \subseteq W$, meaning the differential on $W$ never maps to expressions involving the $z\text{'}s$.

Furthermore, the isomorphism $\alpha: S[X_1] \to W[z_1, \ldots, z_n | \partial(z_i)=0]$ maps each $x_i$ to $z_i$ up to reordering. This means 
\[ \frac{S[X]}{(X_1')} \cong W\]
and since the left hand side is a semi-free extension, this implies $W$ is a semi-free extension of $S$, so that
\[
W[z_1, \ldots, z_n | \partial(z_i)=0] \cong S[Y][z_1,\ldots, z_n| \partial(z_i)=0]
\]
is a minimal semi-free extension of $S$ for some set of variables $Y$.

A common way to construct the homotopy Lie algebra is to consider it as the dual of the $k$-vector space on the variables adjoined in the minimal model, i.e. $\pi^{*}(\varphi) = \{\Sigma k X_i^*\}$, where the Lie bracket is determined from the quadratic part of the differential $\partial^{[2]}$ on $k[X]$, see for example, \cite[appendix B]{Quillen} and \cite[Theorem 4.2]{Avramov84}. More specifically, given

\[
\gamma: \Sigma kX \xrightarrow{\Sigma\partial^{[2]}} \text{Sym}_k^2 (\Sigma kX) \xrightarrow{\delta} \Sigma kX \otimes_k \Sigma kX,
\]
 where $\delta$ is the canonical embedding from the symmetric algebra into the tensor algebra, the Lie bracket is determined by \[\gamma^{*}: (\Sigma kX)^* \otimes _k (\Sigma kX)^* \to  (\Sigma kX)^*,\]
 via the isomorphism $(kX \otimes_k kX)^* \cong (kX)^* \otimes_k (kX)^*$.

 In particular, since there are no variables in $W[z_1, \ldots, z_n]$ that have any $z\text{'}s$ in their images under the differential, then $[z_i^*, - ]=0$ for each $i = \{1, \ldots, n\}$ for $\pi^{\geq2}(\varphi) = \pi^*(\varphi)$. This is precisely what it means for $z_i^*$ to be central.
\end{proof}

We now give an example to show that $\bigwedge I/I^2$ need not split from $\Tor^{R}_*(S,S)$ when $I/I^2$ is not $S$-free. 

\begin{ex}
    Consider the ideal $I=(x^2,xy,xz)$ in the polynomial ring $R=k[x,y,z]$. With a quick Macaulay 2 calculation one can see $\Tor_2(R/I,R/I)$ is generated by the cokernel of the matrix 
 \[   \begin{bmatrix}
        z & y & x & 0 & 0 & 0 & 0 & 0 & 0\\
        0 & 0 & 0 & z & y & x& 0 & 0 & 0 \\
        0 & 0 & 0 & 0 & 0 & 0 & z & y & x \\
    \end{bmatrix},
 \]
    meaning $\Tor_2(R/I,R/I)$ is a $k$-space.

Let $e_1\wedge e_2 , e_1\wedge e_3, e_2\wedge e_3$ generate $\bigwedge^2 I/I^2$. We compute $\bigwedge^2 I/I^2$ as the cokernel of the matrix
 \[   \begin{bmatrix}
        x & 0 & -z & y & 0 & 0 & 0 & 0 \\
        0 & x & y & 0 & 0 & y& z & 0  \\
        0 & 0 & 0 & 0 & y & -x & 0 & z  \\
    \end{bmatrix}.
 \]
This tells us that $ze_1\wedge e_2$ is nonzero in $\bigwedge^2 I/I^2$. Thus, there is no injection from $\bigwedge^2 I/I^2$ to $\Tor_2(R/I,R/I)$, and hence there is no splitting between $\bigwedge I/I^2$ and $\Tor_*^{R}(S,S)$.\end{ex}

We end this section with the question (and answer) of whether there is a larger exterior algebra that splits from $\Tor_{*}^{R}(S,S)$. If we let $I/I^2= F \oplus W$, where $F$ generates a free summand of $I/I^2$, then the question is ``Does $\bigwedge (F \oplus W)$ contain a higher free rank module that splits from $\Tor^{R}_{*}(S,S)$ as $R$ algebras?" 

The following example gives an ideal $I$ where the conormal module $I/I^2$ has no free summands. However, the exterior algebra of the conormal module is isomorphic to the $\Tor$ algebra.

\begin{ex}
    Let $R=k[x,y]$ and $I=(x^2,xy)$. One can check that $\Tor_{n}^{R}(R/I,R/I)=0$ for $n\geq 3$ and for degree reasons $\bigwedge^nI/I^2 =0$ for $n\geq 3$. Furthermore, $\Tor_2^{R}(R/I,R/I)$ is generated by the cokernel of $\begin{bmatrix}
        y & x
    \end{bmatrix}$, which is isomorphic to $\bigwedge^2 I/I^2$. Therefore, $\Tor_{*}^{R}(R/I,R/I) \cong \bigwedge I/I^2$.
\end{ex}

\section{Explicit Maps}
\label{sec:explicit-maps}

In this section, we will define an explicit algebra projection from $\Tor^R_{*}(S,S)$ to $\bigwedge I/I^2$ in lower degrees. We will keep with the general setting of the paper with $R$ a local Noetherian ring, $I$ an ideal of $R$ minimally generated by $a_1, \ldots, a_n$, and $S=R/I$, except we will assume that $I/I^2$ is a free $S$-module for the sake of ease. However, the analogous results hold when we instead consider a free summand of the conormal module. 
In \cref{HLA}, we will mention a connection with the homotopy Lie algebra of $R$. 

Let $R[X]$ be the minimal model of $S$ over $R$. Identifying $S[ X_1 ]= S\langle X_1\rangle$ with $\bigwedge I/I^2$, we construct a concrete map of dg algebras 
\[
\Psi\colon S[ X] \to \bigwedge I/I^2,
\] 
where $\bigwedge I/I^2$ has trivial differential, splitting the natural inclusion $i \colon \bigwedge I/I^2 \to S[ X]$. Using $a_{V^t}$ to denote $a_{i_t}\wedge \cdots \wedge a_{i_1}$ for an ordered list $V^t=i_t,\ldots, i_1$, we can then define the natural inclusion $i$ by sending that sends $\overline{a_{V}}$ to $x_V$ for an ordered set $V \subseteq \{1,\ldots,n\}$ and $x_V \in S\langle X_1 \rangle=\bigwedge (S X_1)$. 

Since $S[X]$ is a free graded commutative algebra over $S$ it is enough to define $\Psi$ on the variables $x\in X$ and extend multiplicatively. We construct $\Psi$ to be a chain map, which reduces to checking the following diagram commutes on the variables. 
\begin{align}
\label{comm-diag2}
\xymatrixcolsep{1.3pc}
\xymatrix{& \cdots \ar[r] & S[ X ]_3 \ar[r]^{\partial} \ar[d]^{\Psi_{3}} & S[ X]_2 \ar[r]^{\partial} \ar[d]^{\Psi_2} & S[ X ]_1 \ar[d]^{\Psi_1} \ar[r]^{\partial} & S[ X ]_0 = S \ar[d]^{\Psi_0} \\
& \cdots \ar[r] &\bigwedge^3 I/I^2 \ar[r]^{0}  & \bigwedge^2 I/I^2 \ar[r]^{0} & I/I^2 \ar[r]^{0} & S\\}
\end{align}

I.e. $\Psi\circ \partial(x)=0$ for each $x\in X$. 

We begin with necessary notation. From now on in this paper, we assume each (finite) set of variables $X_n$ is ordered and denote $x_{j,i}$ to be the $i$th element of degree $j$ in $R[X]$. We also use the condensed notation 
\[
x_{j,V^t}= x_{j,i_t}x_{j,i_{t-1}}\cdots x_{j,i_1}
\]
for an ordered list $V^t=i_t,\ldots,i_1$ of size $t$.

Furthermore, the differential of $x_{j,i}$ for $j>1$ is given by 
\[
\partial(x_{j,i})= \sum_{d_1+\cdots+d_m=j-1} \lambda_{j,(d_1|i_1)\ldots(d_m|i_m)}x_{d_m,i_m}\cdots x_{d_1,i_1},
\] 
where  $d_i \leq d_k$ for $i < k$ and $\lambda_{j,(d_1|i_1)\ldots(d_t|i_t)} \in R$. Note that $\partial(x_{1,i})=a_i$ for $i \in [n]$. We will use a condensed notation for the coefficients of products of the same degree
\[
\lambda_{j,(d_k|i_2)(d_k|i_1)} = \lambda_{j,(d_k|i_2,i_1)}.
\]
Lastly, for later use in defining signs, we set 
\[\sigma(i_j,V)=|\{i \in V : i_j< i\}|.\]

We now begin by defining the map $\Psi$ by starting with the first degree $\Psi_1$ via obvious choice  
\[
\Psi_1: S[X]_1 \to \bigwedge^1 I/I^2= I/I^2 \ \ \ 
{\textrm{ by }} \ \ 
\Psi_1(x_{1,i})=\overline{a_i}.
\]
As we extend the map multiplicatively, this gives
\[
\Psi_1(x_{1,V})=\overline{a_V} \quad \text{for} \quad V\subseteq [n].
\]



In fact, we define a map which we call \[\psi :R[ X ] \to \bigwedge I/I^2\] such that $\psi \circ \partial =0$, which then induces the desired map $\Psi$ since the codomain is annihilated by $I$. 

Before jumping into the general map in low degrees given in Examples \ref{wedge2} and \ref{wedge3}, we provide a smaller example to help the reader follow the notation in the general setting. 
\begin{ex}
    Let $R$ be a ring with an ideal $I$ minimally generated by $a_1,a_2,a_3$ where $I/I^2$ is free on the images $\overline{a_1}, \overline{a_2}, \overline{a_3}$. 

    For this example we will let $R[X]$ be the minimal model for $R/I$, but we assume that $R[X]$ has three degree $1$ variables $x_{1,1} , x_{1,2}, x_{1,3}$, two degree $2$ variables $x_{2,1}, x_{2,2}$, and one degree $3$ variable $x_{3,1}$.

    We first begin defining 
    $$\psi_{1}: R[X]_1 \to \bigwedge^{1} I/I^2 \quad \text{by} \quad\psi_1(x_{1,i})=\overline{a_i}.$$
    We need to verify the chain map condition, so we check $\psi_1(\partial(x_{2,j}))=0$ for $j=1,2$.

    Notice
    $$\partial(x_{2,j}) = \lambda_{j, (1|1)} x_{1,1} + \lambda_{j, (1|2)} x_{1,2} + \lambda_{j, (1|3)} x_{1,3}$$
    for $\lambda_{j, (1|i)} \in R$. Applying the differential once more would imply
    \begin{equation}
\label{small example deg 0}
    0=\partial^2(x_{2,j})= \lambda_{j, (1\mid 1)} a_{1} + \lambda_{j, (1\mid 2)} a_{2} + \lambda_{j, (1 \mid 3)} a_{3}.
\end{equation}

Since $I/I^2$ is free with basis $\overline{a_1},\overline{a_2}, \overline{a_3}$, the image of (\ref{small example deg 0}) in $S$ would imply that the image of $\lambda_{j,(1\mid i)}$ is zero in $S$, or $\lambda_{j,(1 \mid i)} \in I$ in $R$ for  $i = 1,2,3$. Thus, 
\[
\psi_1(\partial(x_{2,j}))= \lambda_{j, (1\mid 1)} \overline{a_{1}} + \lambda_{j, (1\mid 2)} \overline{a_{2}} + \lambda_{j, (1 \mid 3)} \overline{a_{3}} 
\]is zero in $I/I^2$.

In order to define $\psi_2: R[ X]_2 \to \bigwedge^2I/I^2 $, we need decide what the image of each $x_{2,j} \in X_2$ should be. Since each  $\lambda_{j,(1\mid i)} \in I$, we may write each $\lambda_{j,(1\mid i)}$ in terms of our generating set $a_1, a_2, a_3$. Thus, we let
\begin{equation}
\label{smallex deg1_decomp}
\lambda_{j,(1\mid i)}= C_{j,(1\mid 1,i)}a_{1} + C_{j,(1\mid 2,i)}a_{2}  +C_{j,(1\mid 3,i)}a_{3} 
\end{equation}
for some elements $C_{j,(1\mid k,i)}$ in $R$.

We now define $\psi_2$ using the coefficients from (\ref{smallex deg1_decomp}) to determine the image of $x_{2,j}$ under the map:
\begin{align*}
    \psi_2: R[ X]_{2}  &\to \bigwedge^{2} I/I^2\\
    x_{1,V^2} &\mapsto \overline{a_{V^2}}  \\
    x_{2,j} &\mapsto \sum_{V^2}C_{j,(1|V^2)} \,  \overline{a_{V^2}}.
\end{align*}
where $V^2$ is an ordered list of size two in $\{1,2,3\}$.

Note that the first assignment is the multiplicative extension of $\psi_1$. Furthermore, since $\psi(IR[ X]_{2})=0$ we get an induced map from $S[ X]_2$ to $\bigwedge^2 I/I^2$. We now check the chain map condition, i.e. $\psi_2(\partial(x_{3,1})) = 0$. Suppose
\[
\partial(x_{3,1}) = \lambda_{1, (2\mid1)}x_{2,1} + \lambda_{1,(2\mid 2)}x_{2,2} + \lambda_{1,(1\mid 1,2)}x_{1,1}x_{1,2} + \lambda_{1, (1 \mid 1,3)}x_{1,1}x_{1,3}+  \lambda_{1, (1 \mid 2,3)}x_{1,2}x_{1,3}.
\]
Applying $\psi_2$, we see 
\begin{align*}
    \psi_2(\partial(x_{3,1})) &= \lambda_{1,(2\mid 1)}\left[ C_{1,(1 \mid 1,2)}\overline{a_1a_2} + C_{1,(1 \mid 1,3)}\overline{a_1a_3}+ C_{1,(1 \mid 2,3)}\overline{a_2a_3}\right] \\
    &+ \lambda_{1,(2\mid 2)}\left[ C_{2,(1 \mid 1,2)}\overline{a_1a_2} + C_{2,(1 \mid 1,3)}\overline{a_1a_3}+ C_{2,(1 \mid 2,3)}\overline{a_2a_3}\right] \\
    &+ \lambda_{1, (1 \mid 1,2)}\overline{a_1a_2} +\lambda_{1, (1 \mid 1,3)}\overline{a_1a_3}+  \lambda_{1, (1 \mid 2,3)}\overline{a_2a_3}\\
    &= \sum_{i_2<i_1}\left( \lambda_{1,(2 \mid 1)}C_{1,(1 \mid i_2,i_1)}+ \lambda_{1,(2 \mid 2)}C_{2,(1 \mid i_2,i_1)} +\lambda_{1,(1\mid i_2i_1)}\right)\overline{a_{i_2}a_{i_1}}.
\end{align*}
We claim the image of $\left( \lambda_{1,(2 \mid 1)}C_{1,(1 \mid i_2,i_1)}+ \lambda_{1,(2 \mid 2)}C_{2,(1 \mid i_2,i_1)} +\lambda_{1,(1\mid i_2i_1)}\right)  = 0$ in $S$ for  $i_2 <i_1$ and thus $\psi(\overline{\partial}(x_{3,i}))=0$ in $\bigwedge^{2}I/I^2$. We will show this for the coefficient of $\overline{a_1a_3}$ and claim the others follow in a similar fashion (knowing the general proof is to come). Like before, we use the fact that $\partial^2(x_{3,1})=0$, and so one has
\begin{align*}
    0 &= \lambda_{1,(2 \mid 1)}\left(\lambda_{1, (1|1)} x_{1,1} + \lambda_{1, (1|2)} x_{1,2} + \lambda_{1, (1|3)} x_{1,3} \right) + \lambda_{1,(2 \mid 2)}\left( \lambda_{2, (1|1)} x_{1,1} + \lambda_{2, (1|2)} x_{1,2} + \lambda_{2, (1|3)} x_{1,3} \right) \\
    &+ \lambda_{1, (1\mid 1,2)}\left( a_1x_{1,2} - a_2x_{1,1}\right)
    + \lambda_{1, (1\mid 1,3)}\left( a_1x_{1,3} - a_3x_{1,1}\right) + \lambda_{1, (1\mid 2,3)}\left( a_2x_{1,3} - a_3x_{1,2}\right).
\end{align*}
By the observation $\lambda_{j,(1|i)} \in I$ above, we substitute for each $\lambda_{j,(1|i)}$ the expression (\ref{smallex deg1_decomp}). Furthermore, since $\{ x_{1,1},x_{1,2}, x_{1,3}\}$ forms a basis for $R[ X ]_1=RX_1$, the linear independence guarantees that the coefficient of $x_{1,i}$ is $0$ in $R$ for each $i=1,2,3$. Thus, we see that the coefficient of $x_{1,3}$ is
\small{
\[
\lambda_{1,(2|1)}\left( C_{1,(1\mid 3)}a_1 +C_{1,(2\mid 3)}a_2 + C_{1,(3\mid 3)}a_3\right) + \lambda_{1,(2|2)}\left( C_{2,(1\mid 3)}a_1 +C_{2,(2\mid 3)}a_2 + C_{2,(3\mid 3)}a_3\right) + \lambda_{1,(1 \mid 1,3)}a_1 + \lambda_{1,(1 \mid 2,3)}a_2
\]}
which is 0. One may notice, the expression above also lives in $I$. Now consider the image of the expression above in $I/I^2$ as an $S$-linear combination of the basis $\overline{a_1}, \ldots , \overline{a_n}$. By linear independence the coefficients of each $\overline{a_{i}}$ are zero in $S$, and therefore the coefficient of each $a_{i}$ in the expression above lies in $I$. In particular, the coefficient of $a_1$
\[
\lambda_{1,(2|1)} C_{1,(1\mid 3)}+  \lambda_{1,(2|2)} C_{2,(1\mid 3)} + \lambda_{1,(1 \mid 1,3)}
\]
lives in $I$, which is what we wanted to show.
\end{ex}

The next two examples demonstrate how the map $\psi$ can be constructed generally in low degrees. It is easy to see that $\Psi_1: S[ X]_1 \to I/I^2$ induces an isomorphism, so we start with the first non-trivial case. 

\begin{case}
\label{wedge2}

In this case we define $\psi_2: R[ X ]_{2} \to \bigwedge^2I/I^2$ in general and verify that $\psi_{\leq 2} $ is a chain map. 
Using the notation above, we begin by showing $\psi_{1}(\partial(x_{2,j}))=0$ for $x_{2,j}\in X_2$. Notice, 
\[
\partial(x_{2,j}) = \sum_{i} \lambda_{j,(1 \mid i_1)}x_{1,i_1}
\]
for $\lambda_{j,(1\mid i_1)}\in R$. Applying the differential once more would imply
\begin{equation}
\label{deg0}
   0= \partial^2(x_{2,j})= \sum_{i_1}\lambda_{j,(1 \mid i_1 )}a_{i_1}.
\end{equation}
Since $I/I^2$ is free with basis $\overline{a_1},\ldots, \overline{a_n}$, the image of (\ref{deg0}) in $I/I^2$ would imply that the image of $\lambda_{j,(1\mid i_1)}$ is zero in $S$, or $\lambda_{j,(1 \mid i_1)} \in I$ for each $i_1 \in \{1,\ldots, n\}$. Thus, 
\[
\psi(\partial(x_{2,j}))= \sum_{i_1} \lambda_{j,(1\mid i_1)}\overline{a_{i_1}}
\]is zero in $I/I^2$.

We use a similar procedure to produce a map $$\psi_2 : R[ X]_2 \to \bigwedge^2I/I^2 $$ with the property that $\psi_2\left(\partial(x)\right)=0$ for $x\in X_3$. However, it is not clear what the image of each $x_{2,j}\in X_2$ should be under $\psi$.

Since each $\lambda_{j,(1\mid i_1)} \in I$, we may write it in terms of our generating set. Thus, we let
\begin{equation}
\label{deg1_decomp}
\lambda_{j,(1\mid i_1)}= \sum_{i_2} C_{j,(1\mid i_2,i_1)}a_{i_2},
\end{equation}
for some elements $C_{j,(1\mid i_2,i_1)}$ in $R$.

The coefficients in (\ref{deg1_decomp}) become useful in the next line as we define the chain map on $R[ X ]_{2}$. Let $V^{2} \subseteq \{1,\ldots,n\}$ denote any ordered list of length 2, then we can define $\psi_{2}$ by 
\begin{align*}
    \psi_2: R[ X]_{2}  &\to \bigwedge^{2} I/I^2\\
    x_{1,V^2} &\mapsto \overline{a_{V^2}}  \\
    x_{2,j} &\mapsto \sum_{V^2}C_{j,(1|V^2)} \,  \overline{a_{V^2}}.
\end{align*}
Note the first assignment is the multiplicative extension of $\psi_1$ and that $\psi(IR[ X]_{2})=0$, so $\psi$ induces a map from $S[ X]_2$ to $\bigwedge^2 I/I^2$.

One needs to check $\psi_2(\partial(x_{3,k})) = 0$ for all $x_{3,k}\in X_{3}$.
Indeed,
\begin{align*}
    \psi_2 \left( \partial(x_{3,k})\right) &= \sum_{j}\lambda_{k,(2|j)}\sum_{V^2}C_{j,(1|V^2)} \, \overline{a_{V^2}} + \sum_{V^2}\lambda_{k
    , (1|V^2)}\overline{a_{V^2}}\\
    &= \sum_{V^2}\left(\sum_{j}\lambda_{k,(2|j)}C_{j,(1|V^2)} +\lambda_{k, (1|V^2)} \right) \overline{a_{V^2}}.
\end{align*}
We claim the image of $\left(\sum_{j}\lambda_{k,(2|j)}C_{j,(1,V^2)} +\lambda_{k, (1|V^2)} \right)  = 0$ in $S$ for each $V^2$ and thus $\psi(\overline{\partial}(x_{3,i}))=0$ in $\bigwedge^{2}I/I^2$. 
To see this, we use the fact  $\partial^2(x_{3,k})=0$ which gives
\begin{align*}
    &0= \partial \left(\sum_{j} \lambda_{k,(2|j)}x_{2,j} + \sum_{V^2}\lambda_{k, (1|V^2)}x_{1,V^2} \right) \\ 
    &=\sum_{j}\lambda_{k,(2|j)}\left( \sum_{i_1} \lambda_{j,(1|i_1)}x_{1,i_1} \right) + \sum_{V^2}\lambda_{k, (1|V^2)} \left( \sum_{i_v\in V^2}(-1)^{\sigma(i_v,V^2\setminus i_v)} a_{i_v}x_{1,V^2\setminus\{i_v\}}\right).
\end{align*}

By the observation that $\lambda_{j,(1|i_1)} \in I$ above, we substitute   $ \sum_{i_2} C_{j,(1,i_2i_1)}a_{i_2}$ for $\lambda_{j,(1|i_1)}$ by (\ref{deg1_decomp}). Furthermore, since $\{ x_{1,1}, \ldots, x_{1,n}\}$ forms a basis for $R[ X ]_1=RX_1$, the linear independence guarantees the coefficient of $x_{1,i_1}$ is $0$ in $R$ for each $i_1 \in \{1, \ldots, n\}$. Thus,
\begin{equation}
\label{coeff3}
    \sum_{j}\lambda_{k,(2|j)}\sum_{i_2}C_{j,(1|i_2,i_1)}a_{i_2} + \sum_{i_{r}< i_1}\lambda_{k, (1|i_r,i_1)}a_{i_r}- \sum_{i_1<i_r}\lambda_{k, (1|i_r,i_1)}a_{i_r} = 0.
\end{equation}

One may notice, the expression above also lives in $I$. Now consider the image of the expression above in $I/I^2$ as an $S$-linear combination of the basis $\overline{a_1}, \ldots , \overline{a_n}$. By linear independence the coefficients of each $\overline{a_{i_2}}$ are zero in $S$, and therefore the coefficient of each $a_{i_2}$ in the expression above lies in $I$. In particular, for any $i_2 < i_1$ the coefficient of $a_{i_{2}}$ from (\ref{coeff3}) satisfies 
$$
\sum_{j}\lambda_{k,(2|j)}C_{j,(1|i_2,i_1)} + \lambda_{k, (1|i_2,i_1)} \in I.
$$ 

Thus, for an ordered sequence $V^2 \subseteq \{1, \ldots, n\}$, the image of
\begin{equation}
\label{deg2}
\sum_{j}\lambda_{k,(2|j)}C_{j,(1|V^2)} +\lambda_{k, (1|V^2)}
\end{equation}
in $S$ is zero, which proves the claim.
\end{case}

\begin{remark} 
\label{HLA}  Under the assumption that $I/I^2$ is free, we get a partial recovery of \cref{centralrecovery} from the construction in \cref{wedge2}, namely we see how degree $2$ elements commute with degree $2$ elements in the homotopy Lie algebra.  That is, if $R[X]$ is the minimal model of  $\varphi: R \to S$ where $X_1$ are the degree one elements adjoined whose differential gives a set of minimal generators of $I/I^2$, then we can show that $[x_{1,i_2}^*, x_{1,i_1}^*]=0$ for all $i_2,i_1 \in \{1, \ldots, n\}$. 
 
Recall, $\partial(x_{3,k})=\sum_{j} \lambda_{k,(2|j)}x_{2,j} + \sum_{V^2}\lambda_{k, (1|V^2)}x_{1,V^2}$, where $\lambda_{k,(2|j)} \in m$ by definition of the minimal model. This means we know that the bracket on $x_{1,i_2}^*$ and $ x_{1,i_1}^*$ is
 \[
 [x_{1,i_2}^*, x_{1,i_1}^*] = \sum_{k}  \overline{\lambda_{k,(1|i_2,i_1)}} x_{3,k}^*.
 \]
 In Case \ref{wedge2}, we showed $\sum_{j}\lambda_{k,(2|j)}C_{j,(1|i_2,i_1)} + \lambda_{k, (1|i_2,i_1)}\in I \subseteq m$, so since $\lambda_{k,(2|j)} \in m$, we have that $\lambda_{k,(1|i_2,i_1)}$ lives in $m$ as well. Thus,
 \[
 [x_{1,i_2}^*, x_{1,i_1}^*]=0
 \]
for all $i_2,i_1 \in \{1, \ldots, n\}$.
 
\end{remark}
 
\begin{case}
\label{wedge3}
    In this case we show $\psi: R[ X_{\leq3}]\to\bigwedge I/I^2$ is a chain map by expanding upon \cref{wedge2} using the same notation and results. By induction, we only need define $\psi(X_3)$ and show that $\psi(\partial(x_{4,p}))=0$ for all $p$. 

    Let $V^3$ denote any strictly ordered list of length $3$ contained in $\{1,\ldots,n\}$, then
    \[
    \partial(x_{4,p})= \sum_{k}\lambda_{p,(3|k)}x_{3,k} + \sum_{i_r}\sum_{j} \lambda_{p,(1|i_r)(2|j)}x_{1,i_r}x_{2,j}+\sum_{V^3}\lambda_{p,(1|V^3)}x_{1,V^3}.
    \]
    
    In order to apply $\psi$ to $\partial(x_{4,p})$, we need to decide where $x_{3,k}$ maps to. Recall from \cref{wedge2}, we found that the expression (\ref{deg2}) lives in $I$ for each list $V^2\subseteq [n]$ of length $2$. This allows us to write (\ref{deg2}) as a linear combination of the generators of $I$, 
    \begin{equation}
    \label{deg2-decomp}
    \sum_{j}\lambda_{k,(2|j)}C_{j,(1|i_2,i_1)} + \lambda_{k, (1|i_2,i_1)} = \sum_{i_3}C_{(3\mid k)(1 \mid i_3,i_2,i_1)}a_{i_3}.
    \end{equation}

    Similarly to the previous example, we use the new coefficient from (\ref{deg2-decomp}) to define
    \[
    \psi\left( x_{3,k} \right)= \sum_{V^3} C_{(3,k)(1|V^3)}\overline{a_{V^3}}
    \]
    for each $k$ and where $V^3$ denotes any strictly ordered list of length three in $\{1,\ldots, n\}$ and extend $\psi$ on all other degree three elements multiplicatively from $\psi_1$ and $\psi_2$.
    
    Applying $\psi$ to $\partial(x_{4,p})$ we have 
\begin{align*}
\psi(\partial(x_{4,p})) &=\sum_{k}\sum_{V^3}\lambda_{p,(3|k)}C_{k,(1|V^3)}\overline{a_{V^3}} + \sum_{i_r}\sum_{j} \lambda_{p, (1|i_r)(2|j)}\overline{a_{i_r}}\sum_{V^2}C_{j,(1|V^2)} \overline{a_{V^2}}+ \sum_{V^3} \lambda_{p,(1|V^3)}\overline{a_{V^3}}\\
&=\left( \sum_{k}\lambda_{p,(3|k)}C_{k,(1|V^3)} + \sum_{j}\sum_{i_r \sqcup V^2 = V^3} (-1)^{\sigma(i_r,V^2
)}\lambda_{p, (1|i_r)(2|j)}C_{j,(1|V^2)} + \lambda_{p,(1|V^3)}\right)\overline{a_{V^3}}.
\end{align*}
    For each $V^3$, we claim the coefficient of $\overline{a_{V^3}}$ from above is $0$ in $S$.

With a similar approach to the previous example, we apply the differential to $x_{4,p}$ twice. 
\begin{multline}
\label{part24}
    \partial^2(x_{4,p}) = \sum_{k} \lambda_{p,(3|k)} \left( \sum_{j} \lambda_{k,(2|j)}x_{2,j} + \sum_{V^2} \lambda_{k,(1| V^2)} x_{1,V^2}\right)\\
    + \sum_{i_r}\sum_{j} \lambda_{p,(1|i_r)(2|j)}\left(a_{i_r} x_{2,j}- x_{1,i_r} \sum_{i_1} \lambda_{j,(1|i_1)}x_{1,i_1}  \right) + \sum_{\substack{V^3 \\ i_r \sqcup V^2 = V^3}} \lambda_{p,(1|V^3)} (-1)^{\sigma(i_r,V^2)}a_{i_r}x_{1,V^2}.
    \end{multline}
Since $\partial^{2}(x_{4,p})=0$ and $R[ X ]_{2}$ is free, we know that the coefficient of each degree $2$ monomial in \cref{part24} is also zero in $R$. Without loss of generality, we let $V_1^3=\{i_3,i_2,i_1\}$ and $V^2_1=\{i_2,i_1\}$ to demonstrate the exact procedure. Notice the coefficient of each $x_{2,j}$ is 
\begin{equation}
\label{coeff2}
\sum_{k}\lambda_{p,(3|k)}\lambda_{k,(2|j)} + \sum_{i_r} \lambda_{p,(1|i_r)(2|j)}a_{i_r},
\end{equation}
while the coefficient of $x_{1,V_1^2}$ is 
\begin{equation}
\label{coeff1}
\sum_{k} \lambda_{p,(3|k)}\lambda_{k,(1|V_1^2)}-\sum_{j}\sum_{i_r\sqcup i_l=V_1^2}(-1)^{\sigma(i_r,i_l)}\lambda_{p,(1|i_r)(2|j)}\lambda_{j,(1|i_l)} + \sum_{\substack{V^3_1 \\ i_r\sqcup V_1^2= V^3_1}}(-1)^{\sigma(i_3,V^2_1)}\lambda_{p,(1|V^3_1)}a_{i_r},
\end{equation}
so both expressions are $0$.
We now utilize the previous expression \cref{deg2}. First we multiply the coefficients of each $x_{2,j}$ in \cref{coeff2} by $C_{j,(1|V^2_1)}$, respectively, and add the sum of these coefficients to the sum of those in \cref{coeff1}. With some rearranging, we get the expression
\begin{multline}
\label{deg4ex}
    \sum_{k}\lambda_{p,(3|k)}\left( \sum_j\lambda_{k,(2|j)}C_{j,(1|V^2_1)} + \lambda_{k,(1|V^2_1)}\right) + \sum_j\sum_{i_r} \lambda_{p,(1|i_r)(2,j)}C_{j,(1|V^2_1)}a_{i_r}\\
    -\sum_j\sum_{i_r\sqcup i_l=V_1^2}(-1)^{\sigma(i_r,i_l)}\lambda_{p,(1|i_r)(2|j)}\lambda_{j,(1|i_l)} + \sum_{\substack{V^3_1 \\ i_r\sqcup V_1^2= V^3_1}}(-1)^{\sigma(i_r,V^2_1)}\lambda_{p,(1|V^3_1)}a_{i_r},
\end{multline}
which is equal to $0$ in $R$.

 Now, recall the expression in (\ref{deg2}) is in $I$. Thus, there is a decomposition as shown in (\ref{deg2-decomp})
\[
\sum_j\lambda_{k,(2|j)}C_{j,(1|V^2_1)} + \lambda_{k,(1|V^2_1)}= \sum_{i_r}C_{k,(1|i_rV^2_1)}a_{i_r}.
\]
Similarly, we showed $\lambda_{j,(1|i_l)}$ lives in $I$ as well for which we wrote 
\[
\lambda_{j,(1|i_l)}=\sum_{i_t}C_{j,(1|i_ti_l)}a_{i_t}.
\]
Substituting each decomposition into (\ref{deg4ex}) we obtain an expression that can be written as a linear combination of the generators of $I$ and is equal to $0$:
\begin{multline*}
    \sum_{k}\lambda_{p,(3|k)}\left( \sum_{i_r}C_{k,(1|i_rV^2_1)}a_{i_r}\right) + \sum_j\sum_{i_r} \lambda_{p,(1|i_r)(2,j)}C_{j,(1|V^2_1)}a_{i_r}\\
    -\sum_j\sum_{i_r\sqcup i_l=V_1^2}(-1)^{\sigma(i_r,i_l)}\lambda_{p,(1|i_r)(2|j)}\sum_{i_t}C_{j,(1|i_ti_l)}a_{i_t} + \sum_{\substack{ i_r\sqcup V_1^2= V^3_1}}(-1)^{\sigma(i_r,V^2_1)}\lambda_{p,(1|V^3_1)}a_{i_r}.
\end{multline*}
Now consider the image of the expression above in $I/I^2$ as an $S$-linear combination of the basis $\overline{a_1}, \ldots , \overline{a_n}$. By independence the coefficients of each $\overline{a_{i_r}}$ are zero in $S$, and therefore we can isolate the coefficient of $a_{i_3}$ shown below, which must live in $I$
\begin{align*}
&\sum_k \lambda_{p,(3|k)}C_{k,(1|V_1^3)} + \sum_{j}\lambda_{p,(1|i_3)(2|j)}C_{j,(1|V_1^2)} -\sum_{i_r \sqcup i_l} \sum_{j} (-1)^{\sigma(i_r,i_l)}\lambda_{p,(1|i_r)(2|j)}C_{j,(1|i_3i_l)}+\lambda_{p,(1|V_1^3)}\\
&= \sum_{k}\lambda_{p,(3|k)}C_{k,(1|V^3_1)} + \sum_{j}\sum_{i_r \sqcup V^2 = V_1^3} (-1)^{\sigma(i_r,V^2
)}\lambda_{p, (1|i_r)(2|j)}C_{j,(1|V^2)} + \lambda_{p,(1|V_1^3)}.
\end{align*}
Thus, for any ordered sequence $V^3 \subseteq [n]$, the image of  
\begin{equation}
\label{deg3}
\sum_{k}\lambda_{p,(3|k)}C_{k,(1|V^3)} + \sum_{j}\left(\sum_{i_r \sqcup V^2 = V^3} (-1)^{\sigma(i_r,V^2
)}\lambda_{p, (1|i_r)(2|j)}C_{j,(1|V^2)}\right) + \lambda_{p,(1|V^3)}
\end{equation}
is zero under $\psi$, which is what we needed to show.

Since we have shown (\ref{deg3}) lives in $I$, then as we have done before, we can write $(\ref{deg3})$ as a linear combination of the generators in $I$.
\begin{equation}
    \sum_{k}\lambda_{p,(3|k)}C_{k,(1|V^3)} + \sum_{j} \left(\sum_{i_r \sqcup V^2 = V^3} (-1)^{\sigma(i_r,V^2
)}\lambda_{p, (1|i_r)(2|j)}C_{j,(1|V^2)} \right)+ \lambda_{p,(1|V^3)} = \sum_{i_4}C_{(1|i_4,V^3)}a_{i_4},
\end{equation}
which we would then use to define $\psi$ on degree $4$ variables.

\end{case}

We can continue this process iteratively to define $\psi $ for each variable $x_{t,j}$ of degree $t$, using the expressions from each lower degree case to prove that the relevant linear combinations live in $I$, yielding that $\psi(\partial(x))=0$ for all $x \in X$ as desired to obtain a chain map. However, the expressions become unwieldy to write down, so we end here.

\section{Acknowledgements}
Special thanks to Claudia Miller for all of her guidance on this project. Thanks also to Ben Briggs, Srikanth Iyengar, Graham Leuschke, and Josh Pollitz for helpful discussions and suggestions. Lastly, a big thank you to Kory Pollicove and Dorian Kalir for keeping me sane through the weeds. Several results in this paper were first conjectured based on examples computed using \textit{Macaulay2} \cite{M2}.

\nocite{*}
\printbibliography

\end{document}